\documentclass[11pt,letterpaper]{article}

\usepackage{setspace}
\usepackage{amssymb,amsmath,amsthm,tabu,longtable}
\usepackage{subfigure}
\usepackage{calc}
\usepackage{verbatim}
\usepackage{enumerate}
\usepackage{epsfig}
\usepackage{graphicx}
\usepackage{graphics}
\usepackage {tikz}
\usetikzlibrary{automata,arrows,calc,positioning}
\usetikzlibrary{decorations.pathreplacing}
\usepackage{xcolor}
\usepackage{chngcntr}
\counterwithin{figure}{section}
\newtheorem{defn}{Definition}[section]
\newtheorem{lemma}[defn]{Lemma}
\newtheorem{prop}[defn]{Proposition}
\newtheorem{cor}[defn]{Corollary}

\newtheorem{problem}{Problem}

\newtheorem{conj}[defn]{Conjecture}

\newtheorem{theorem}[defn]{Theorem}

                                {\null\hfill$\Box$\par\medskip\medskip\medskip\medskip}

\newcommand*{\QED}{\null\hfill$\Box$\par\medskip}%

\makeatletter
\renewcommand{\pod}[1]{\mathchoice
  {\allowbreak \if@display \mkern 18mu\else \mkern 8mu\fi (#1)}
  {\allowbreak \if@display \mkern 18mu\else \mkern 8mu\fi (#1)}
  {\mkern4mu(#1)}
  {\mkern4mu(#1)}
}

\title{Independence Equivalence Classes of Paths and Cycles}
\author{Iain Beaton\\
\small Department of Mathematics \& Statistics\\[-0.8ex]
\small Dalhousie University\\[-0.8ex] 
\small Halifax, CA\\
\small\tt iain.ac.beaton@gmail.com\\
\and
Jason I. Brown\thanks{Supported by NSERC grant 170450-2013}\\
\small Department of Mathematics \& Statistics\\[-0.8ex]
\small Dalhousie University\\[-0.8ex] 
\small Halifax, CA\\
\small\tt Jason.Brown@dal.ca\\
\and
Ben Cameron\thanks{Corresponding author}\\
\small Department of Mathematics \& Statistics\\[-0.8ex]
\small Dalhousie University\\[-0.8ex]
\small Halifax, CA\\
\small\tt cameronb@mathstat.dal.ca
}

\begin{document}

\tikzset{bignode/.style={minimum size=3em,}}
\maketitle

\begin{abstract}
The independence polynomial of a graph is the generating polynomial for the number of independent sets of each size. Two graphs are said to be \textit{independence equivalent} if they have equivalent independence polynomials. We extend previous work by showing that independence equivalence class of every odd path has size 1, while the class can contain arbitrarily many graphs for even paths. We also prove that the independence equivalence class of every even cycle consists of two graphs when $n\ge 2$ except the independence equivalence class of $C_6$ which consists of three graphs. The odd case remains open, although, using irreducibility results from algebra, we were able show that for a prime $p \geq 5$ and $n\ge 1$ the independence equivalence class of $C_{p^n}$ consists of only two graphs.
\end{abstract}

\setstretch{1.4}

\section{Introduction}\label{sec:intro}

A subset of vertices of a (finite, undirected and simple) graph $G$ is called {\em independent} if the subset induces a subgraph with no edges (the \emph{independence number} of  $G$ is the size of the largest independent set in $G$ and is denoted by $\alpha(G)$, or just $\alpha$ if the graph is clear from context). The {\em independence polynomial} of $G$, denoted by $i(G,x)$ is defined by 
\[ i(G,x)=\sum_{k=0}^{\alpha}i_kx^k,\] 
where $i_k$ is the number of independent sets of size $k$ in $G$. Research on the independence polynomial has been very active since it was first defined in 1983 \cite{INDFIRST,INDPOLY,Chudnovsky2007,Alavi,BDN2000,INDROOTS}
%The graphs $C_n$ and $P_n$ are the cycle on $n$ vertices and the path on $n$ vertices respectively. 

We say that two unlabelled graphs $G$ and $H$, are \textit{independence equivalent}, denoted $G \sim H$, if they have the same independence polynomial. Independence equivalence is clearly an equivalence relation, so we define the \textit{independence equivalence class} of a graph $G$, denoted $[G]$, to be the set of all graphs that are independence equivalent to $G$. If a graph is the only graph in its independence equivalence class, we call this graph \textit{independence unique}. As an example, $P_4$ and $K_3\cup K_1$, both of which have independence polynomial $1+4x+3x^2$, are independence equivalent. On the other hand, each complete graph $K_n$, is independence unique as it is the only graph with independence polynomial $1+nx$. As a graph is completely determined by its independent sets of cardinality $2$, (that is, the non-edges of the graph), it is interesting to see what information is encoded when we do not have access to all of the independent sets but only the information encoded by the independence polynomial (that is, only how many of each size there are). As we have seen, the combinatorial information about the independent sets is not enough to completely distinguish a graph (it cannot even determine whether a a graph is connected). In fact Makowsky and Zhang \cite{Makowksy2017} showed the proportion of independence unique graphs to graphs tends to zero as the order (that is, the number of vertices) tends to infinity. Independence uniqueness and independence equivalence is also of interest in analogy to the corresponding notion for the  chromatic polynomial, the \emph{chromaticity} of a graph (see  chapters 4,5, and 6 of \cite{DongKohTeo2005}). 
%Recent examples of work done on chromaticity include \cite{LauPeng2009,Tomescu2007,ChenSuYao2012,Wangetal2013,KarimHasniLau2016,KarimHasniLau2017}. 
%Most of these results involve determining chromatic equivalence classes or proving that families of graphs are chromatically unique. 
In \cite{Makowsky2014} the authors consider equivalence and uniqueness of a general polynomial arising from a graph, and they also raise the point that one reason to study graph polynomials is to help distinguish non-isomorphic graphs.  

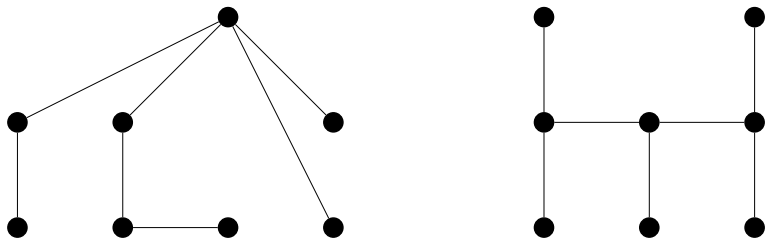
\begin{figure}[htp]
\def\c{0.7}
\def\r{2}
\centering
\scalebox{\c}{
\begin{tikzpicture}
\begin{scope}[every node/.style={circle,thick,draw,fill}]
    \node (1) at (0*\r,0*\r) {};
    \node (2) at (1*\r,0*\r) {};
    \node (3) at (2*\r,0*\r) {};
    \node (4) at (2*\r,1*\r) {};
    \node (5) at (0*\r,-1*\r) {};
    \node (6) at (0*\r,1*\r) {};
    \node(7) at (1*\r,-1*\r) {};
    \node (8) at (2*\r,-1*\r) {};   
\end{scope}

\begin{scope}
    \path [-] (1) edge node {} (2);
    \path [-] (2) edge node {} (3);
    \path [-] (4) edge node {} (3);
    \path [-] (1) edge node {} (5);
    \path [-] (1) edge node {} (6);
    \path [-] (2) edge node {} (7);
    \path [-] (3) edge node {} (8);
\end{scope}

\begin{scope}[every node/.style={circle,thick,draw,fill}]
    \node (9) at (-5*\r,0*\r) {};
    \node (10) at (-4*\r,0*\r) {};
    \node (11) at (-3*\r,-1*\r) {};
    \node (12) at (-2*\r,0*\r) {};
    \node (13) at (-2*\r,-1*\r) {};
    \node (14) at (-5*\r,-1*\r) {};
    \node (15) at (-4*\r,-1*\r) {};
    \node (16) at (-3*\r,1*\r) {};   
\end{scope}

\begin{scope}
    \path [-] (9) edge node {} (16);
    \path [-] (10) edge node {} (16);
    \path [-] (12) edge node {} (16);
    \path [-] (16) edge node {} (13);
    \path [-] (10) edge node {} (15);
    \path [-] (15) edge node {} (11);
    \path [-] (9) edge node {} (14);
\end{scope}

\end{tikzpicture}}
\caption{Independence equivalent trees on $8$ vertices.}%
\label{fig:equivnonwctrees}%
\end{figure}

Returning to independence, in \cite{Stevanovic1997}, Stevanovic showed threshold graphs are independence unique among threshold graphs, doing so from the clique polynomial point of view.  
There is work done by Brown and Hoshino \cite{BrownHoshino2012} that provides a full characterization of independence unique circulant graphs and in the process determines some constructions to obtain graphs that are independence equivalent to circulant graphs. In \cite{Levit2008}, Levit and Mandrescu showed well-covered spiders are independence unique among well-covered graphs. 

However, even for the path $P_{n}$ and cycle $C_{n}$ of order $n$, determining the independence equivalence classes is tricky and subtle (much more so than for other graph polynomials). Chism \cite{Chism2009} showed that  $[P_{2n}]$ contains a few families of graphs (we will expand upon in Section~\ref{sec:paths}) (Zhang \cite{Zhang2012} proved  the same results via different techniques). In \cite{Li2016}, the authors showed the only tree in $[P_n]$ is $P_n$ itself. Most recently, Oboudi \cite{Oboudi2018} completely determined all connected graphs in the independence equivalence classes of cycles. In this work, we extend the results of Oboudi \cite{Oboudi2018} and Li \cite{Li2016} by considering which disconnected graphs can be in $[P_n]$ and $[C_n]$ respectively.

This paper is structured as follows: Section~\ref{sec:paths} is devoted to exploring $[P_n]$. For odd $n$ we show that $P_n$ is independence unique, whereas for even $n$ there can be arbitrarily many nonisomorphic graphs in $[P_n]$. In Section~\ref{sec:cycles}, we consider $[C_n]$, using very different methods depending on the parity of $n$. We find that when $n$ is even (and $n\neq 6$), or a prime power where the base is at least $5$, then $[C_n]=\{C_n,D_n\}$ where $D_n$ is the graph obtained by gluing a leaf of $P_{n-3}$ to one vertex of a triangle (see Figure~\ref{fig:Ln}). Our results for paths and even cycles involve combinatorial analysis that comes from analyzing the coefficients. Our results for prime cycles and prime power cycles, however, are proved using algebraic results by examining the reducibility of the polynomials.

\section{Independence Equivalence Classes of Paths}\label{sec:paths}
The independence equivalence class of a path has been considered before in \cite{Zhang2012,Chism2009}, where the it was shown that there are at least $2$ disconnected graphs in $[P_{2n}]$ and in \cite{Li2016} where it was shown that the only connected graph in $[P_n]$ is $P_n$ itself. For independence polynomials, the highly structured nature of paths allows for an explicit formula for paths:
\begin{theorem}[Arocha, \cite{Arocha1984}] The independence polynomial of a path of order $n$ is given by
\label{thm:PathPoly}

$$i(P_n,x) = \sum_{j=0}^{\lfloor \frac{n+1}{2} \rfloor} \binom{n+ 1 -j}{j}x^j.$$
\QED
\end{theorem} 

Recently, Li, Liu, and Wu \cite{Li2016} completely classified all {\it connected} graphs in $[P_n]$ for all $n$

\begin{theorem}[\cite{Li2016}]\label{cor:connectedpathunique}
For any connected graph $G$ and $n \in \mathbb{N}$, if $i(G,x)=i(P_n,x)$ then $G \cong P_n$.
\QED
\end{theorem}

    However, independence equivalence does not necessarily put a restriction on connectivity. In this section we will consider what disconnected graphs can belong to $[P_n]$. 
    %For odd paths, we completely answer the question and the situation is no different than the connected case. 
    We start by showing that even paths are very different in the disconnected case. We will show that we can have arbitrarily many graphs in the independence equivalence classes of even paths. To do this, we build on the basic results in \cite{Chism2009,Zhang2012} that provides an example of a disconnected graph in $[P_n]$ for even $n$.

\begin{prop}[\cite{Chism2009,Zhang2012}]\label{prop:chismpath}
$P_{2n}\sim P_{n-1}\cup C_{n+1}$ for $n\ge 2$.
\QED
\end{prop}

\begin{figure}[htp]
\def\c{0.7}
\def\r{2}
\centering
\scalebox{\c}{
\begin{tikzpicture}
\begin{scope}[every node/.style={circle,thick,fill,draw}]
    \node[label=above:$u_2$,draw] (2) at (1.13397*\r,0.5*\r) {};
    \node[label=above:$u_3$,draw] (3) at (2*\r,0*\r) {};
    \node[label=above:$u_{n-1}$,draw] (4) at (4*\r,0*\r) {};
    \node[label=above:$u_n$,draw] (5) at (5*\r,0*\r) {};
    \node[label=below:$u_1$,draw] (7) at (1.13397*\r,-0.5*\r) {};
    \node[label=above:$u_4$,draw] (8) at (3*\r,0*\r) {};   
\end{scope}

\begin{scope}

    \path [-] (2) edge node {} (3);
    \path [-] (4) edge node {} (5);
    \path [-] (3) edge node {} (7);

    \path [-] (2) edge node {} (7);
    \path [-] (3) edge node {} (8);
    
\end{scope}

\path (8) -- node[auto=false]{\ldots} (4);
\path [-] (8) edge node {} (3*\r+.7,0*\r) ;
\path [-] (4*\r-.7,0*\r) edge node {} (4);
\end{tikzpicture}}
\caption{The graph $D_n$}%
\label{fig:Ln}%
\end{figure}
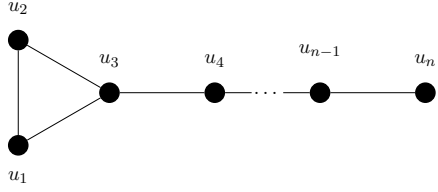

\begin{prop}[\cite{Chism2009,Zhang2012}]\label{prop:chismcycle}
$C_n\sim D_n$ for $n\ge 3$ (where $D_n$ is formed from a triangle by adding a pendant path  -- see Figure~\ref{fig:Ln}).
\QED
\end{prop}

\begin{prop}\label{prop:largeevenpaths}
For any $K \geq 0$, there is an even path whose independence equivalence class has cardinality at least $K$.
\end{prop}
\begin{proof}
Let $N$ be a positive integer, and set $n=2^{\lceil N/2\rceil+2}-2$. We claim that $P_{n}$ has at least $\frac{n}{2}$ non-isomorphic graphs in its independence equivalence class. 
From Proposition~\ref{prop:chismpath}, $P_{2^{\lceil N/2\rceil+2-k}-2}$ is equivalent to $P_{2^{\lceil N/2\rceil+1-k}-2}\cup C_{2^{\lceil N/2\rceil+1-k}}$ for all $k=0,1,\ldots,\lceil N/2\rceil-1$. Therefore, by iteratively applying

\begin{align}
P_{n}\sim  P_{2^{\lceil N/2\rceil+1-k}-2}\cup\bigcup_{\ell=0}^{k} C_{2^{\lceil N/2\rceil+1-\ell}}\label{eq:patheqclass}.
\end{align}

By Proposition~\ref{prop:chismcycle}, for $0\le \ell\le k$, $C_{2^{\lceil N/2\rceil+1-\ell}}\sim D_{2^{\lceil N/2\rceil+1-\ell}}$. Therefore, for each value of $k$, the cycles in (\ref{eq:patheqclass}) can be replaced by equivalent graphs in $2^{k+1}$ ways. This, together with the graph $P_n$, gives $1+2+2^2+\cdots 2^{\lceil N/2\rceil}=2^{\lceil N/2\rceil+1}-1=\frac{n}{2}$ many distinct graphs in $[P_n]$. 
\end{proof}

The surprising difference between the disconnected and connected graphs that are independence equivalent to even paths begs the question of what happens with odd paths. In the odd case, we completely characterize $[P_{2n+1}]$ for all $n$ by showing, in stark contrast to Proposition~\ref{prop:largeevenpaths}, that $P_{2n+1}$ is independence unique for all $n\ge 0$.

\begin{theorem}\label{thm:oddpathsunique}
$P_{2n+1}$ is independence unique for all $n\ge 0$.
\end{theorem}
\begin{proof}
Suppose that there exists a graph $G$ such that $G\sim P_{2n+1}$. Note that $i(P_{2n+1},x)$ is monic for every $n\ge 0$, since there is exactly one independent set of maximum size, $n+1$, by taking a leaf and then every other vertex along the path. So $i(G,x)$ must be monic. Therefore, $G$ must have exactly one independent set of size $n+1$; call this set $S$. If there is a vertex in $V(G)-S$ that is not adjacent to at least two vertices in $S$, then we can take this vertex and $n$ vertices in $S$ that are not adjacent with it to make a second independent set of size $n+1$, a contradiction. Therefore every vertex in $V(G)-S$ is adjacent to at least $2$ vertices in $S$, requiring at least $2n$ edges between $V(G)-S$ and $S$. From the second coefficient of $i(P_{2n+1},x)$, we know that $G$ has exactly $2n$ edges and therefore $G$ is a bipartite graph with bipartition $(V(G)-S,S)$. Therefore, $G$ is triangle-free.

\begin{figure}[!h]
\centering
\scalebox{1}{
\begin{tikzpicture}
\draw (0,0.9) ellipse (2 and 0.75);
\node[text width=1cm] at (0,1) {\Large $S$};

\node[shape=circle,draw=black,fill=black] (1) at (-1.4,-0.75) {};
\node[shape=circle,draw=black,fill=black] (2) at (-0.7,-0.75) {};
\node[shape=circle,draw=black,fill=black] (3) at (0,-0.75) {};
\node[shape=circle,draw=black,fill=black] (4) at (1.4,-0.75) {};

\begin{scope}

    \path [-] (1) edge node {} (-1.5,0.5);
    \path [-] (1) edge node {} (-0.9,0.5);
    
    \path [-] (2) edge node {} (-0.8,0.5);
    \path [-] (2) edge node {} (-0.2,0.5);
    
    \path [-] (3) edge node {} (-0.1,0.5);
    \path [-] (3) edge node {} (0.5,0.5);
    
    \path [-] (4) edge node {} (0.9,0.5);
    \path [-] (4) edge node {} (1.5,0.5);
    
    \path (3) -- node[auto=false]{\ldots} (4);

\end{scope}

\draw[black, decoration={brace, raise=5pt, mirror, amplitude=3mm}, decorate] (-1.5,-1) -- (1.5,-1);

\node[text width=1cm] at (0.4,-1.75) {$n$};
\end{tikzpicture}}

\caption{$G$}%
\label{fig:G}%
\end{figure}

If $G\not\cong P_{2n+1}$, then from Corollary~\ref{cor:connectedpathunique} we know that $G$ must be disconnected. Let $G_1,G_2,\ldots,G_k$ be the connected components of $G$ for some $k\ge 2$. Let $S_i=S\cap V(G_i)$ and $D_i=V(G_i)-S_i$ for $i=1,2,\ldots,k$. Each $G_i$ is bipartite with bipartition $(S_i,D_i)$. Suppose that for some $i$, $|S_i|\le |D_i|$. Now, $\bigcup_{j\neq i}S_j\cup D_i$ is an independent set with at least $n+1$ vertices in it, which contradicts $i(G,x)$ being monic and of degree $n+1$. Therefore, $|S_i|\ge |D_i|+1$ for $i=1,2,\ldots,k$. Therefore, 
$$2n+1=|V(G)|=\sum_{i=1}^k|V(G_i)|=\sum_{i=1}^k\left(|S_i|+|D_i|\right)\ge \sum_{i=1}^k\left( 2|D_i|+1 \right) =2n+k\ge 2n+2,$$
a contradiction. Therefore, $G$ must be connected, and by Corollary~\ref{cor:connectedpathunique}, $G\cong P_{2n+1}$. Therefore $P_{2n+1}$ is independence unique.
\end{proof}

It is interesting to note the contrast between the independence equivalence classes of even and odd paths respectively given by Proposition~\ref{prop:largeevenpaths} and Theorem~\ref{thm:oddpathsunique}. It seems that the key distinction between the independence equivalence class of odd and even paths is the number of independent sets of maximum size. An even path on $n$ vertices has $\frac{n}{2}+1$ maximum independent sets, while an odd path has only one. As seen in the proof of Theorem \ref{thm:oddpathsunique}, a graph having few maximum independent sets determines some structure. We will use a similar approach in the next section for even cycles.

%\newpage
\section{Independence Equivalence Class of Cycles $C_{n}$}
\label{sec:cycles}
An early result in chromaticity is that cycles are chromatically unique \cite{ChaoWhitehead1978}. Clearly this is not the case for independence polynomials as Proposition \ref{prop:chismcycle} shows $C_n \sim D_n$ for $n \geq 3$. In this section, we will show that $[C_n]=\{C_n,D_n\}$ for $n$ even, or $n$ a prime at least $5$ to any power. Along with these results, we have used the computational tools of nauty \cite{McKay2014} and Maple to show that $[C_n]=\{C_n,D_n\}$ for $1\le n\le 32$ with the exceptions of $C_6,C_9,$ and $C_{15}$.  We will present the independence equivalence classes of each of these three exceptional graphs as we proceed. 

Like paths, all connected graphs which are independence equivalent to cycles have been determined.

\begin{theorem}[\cite{Oboudi2018}]\label{thm:cycleconnectedequivclass}
For $n\ge 3$, if $G$ is a connected graph such that $i(G,x)=i(C_n,x)$, then $G\cong C_n$ or $G\cong D_n$.
\QED
\end{theorem}

Given Theorem \ref{thm:cycleconnectedequivclass}, we need only consider disconnected graphs to determine $[C_n]$. We will use an argument on the degree sequence to show that there are no disconnected graphs in $[C_{2n}]$ for $n \geq 2$, and one disconnected graph in $[C_6]$. As is shown in the next theorem, using the principle of inclusion-exclusion, some information about the degree sequence of a graph is encoded in the coefficient of $x^3$ in its independence polynomial.

\begin{theorem}
\label{thm_i3}

 For any graph $G=(V,E)$ with $n$ vertices and $m$ edges
 
 $$i_3(G) = \binom{n}{3} - m(n-2) + \sum_{v \in V} \binom{\deg (v)}{2}-n(C_3),$$
 
 \noindent where $i_3(G)$ is the number of independent sets in $G$ with cardinality three and $n(C_3)$ is the number of $3$-cycles in $G$.
\end{theorem}

\begin{proof}

It is sufficient to show the number of 3-subsets which are not independent is $m(n-2) - \sum_{v \in V} {\deg (v) \choose 2}+n(C_3)$. Any  3-subset of $V$ induces one of the following subgraphs:

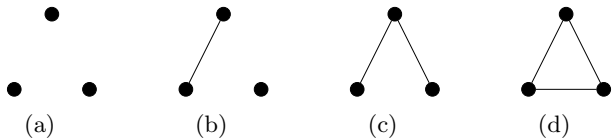
\begin{figure}[!h]
\def\c{0.5}
\centering
\subfigure[]{
\scalebox{\c}{
\begin{tikzpicture}
\node[shape=circle,draw=black,fill=black] (1) at (-1,0) {};
\node[shape=circle,draw=black,fill=black] (2) at (0,2) {};
\node[shape=circle,draw=black,fill=black] (3) at (1,0) {};

\end{tikzpicture}}}
\qquad
\subfigure[]{
\scalebox{\c}{
\begin{tikzpicture}
\node[shape=circle,draw=black,fill=black] (1) at (-1,0) {};
\node[shape=circle,draw=black,fill=black] (2) at (0,2) {};
\node[shape=circle,draw=black,fill=black] (3) at (1,0) {};

\path[] (1) edge node {} (2);
\end{tikzpicture}}}
\qquad
\subfigure[]{
\scalebox{\c}{
\begin{tikzpicture}
\node[shape=circle,draw=black,fill=black] (1) at (-1,0) {};
\node[shape=circle,draw=black,fill=black] (2) at (0,2) {};
\node[shape=circle,draw=black,fill=black] (3) at (1,0) {};

\path[] (1) edge node {} (2);
\path[] (2) edge node {} (3);

\end{tikzpicture}}}
\qquad
\subfigure[]{
\scalebox{\c}{
\begin{tikzpicture}
\node[shape=circle,draw=black,fill=black] (1) at (-1,0) {};
\node[shape=circle,draw=black,fill=black] (2) at (0,2) {};
\node[shape=circle,draw=black,fill=black] (3) at (1,0) {};

\path[] (1) edge node {} (2);
\path[] (2) edge node {} (3);
\path[] (1) edge node {} (3);

\end{tikzpicture}}}
%\caption{Induced subgraphs of $T$ on three vertices}%
%\label{fig:ExEdgeColour}%
\end{figure}

We can construct each non-independent 3-subset by taking an edge $uv$ and a vertex $w$ not incident to the edge. As $G$ has $m$ edges there we will construct $m(n-2)$ subsets. If $w$ is not adjacent to $u$ nor $v$ then we induce the subgraph $(b)$ and construct it once. If $w$ is adjacent to $u$ (or $v$) then we induce the subgraph $(c)$. However this 3-subset will have been constructed in two ways: The edge $uv$ and vertex $w$, and the edge $uw$ (or $vw$) and vertex $v$. Therefore we have counted each 3-subset which induces a subgraph of type $(c)$ twice and $(d)$ three times. 

We can construct each 3-subset which induces a subgraph of type $(c)$ by taking a vertex and choosing any two of its neighbours. Hence there are $\sum_{v \in V} {\deg (v) \choose 2}$ such subsets. Note this counts the number of 3-subsets which induces subgraph $(d)$ three times as well. Clearly the number of 3-subsets which induces subgraph $(d)$ is $n(C_3)$. Thus the number of non-independent 3-subsets is $m(n-2) - \sum_{v \in V} {\deg (v) \choose 2}+n(C_3)$.

\end{proof}

\begin{lemma}
\label{lem:eq3}
Let $n \geq 4$ and $G$ be a graph with $n(C_3)$ many $3$-cycles and $g_i$ many vertices of degree $i$. If $G \sim C_n$ then

\begin{enumerate}[(i)]
\item $\sum\limits_{i=0}^{n-1} g_i = n$,
\item $\sum\limits_{i=1}^{n-1} i \cdot g_i = 2n$,
\item $\sum\limits_{i=2}^{n-1} {i \choose 2}g_i =n + n(C_3)$, and
\item $n(C_3) \geq g_0 + \sum\limits_{i=3}^{n-1} g_i$, that is there are at most $n(C_3)$ vertices not of degree one or two.
\end{enumerate}
\end{lemma}

\begin{proof}
Suppose $G$ is a graph such that $G \sim C_n$. Then $G$ has $n$ vertices and $n$ edges making (i) and (ii) trivial. To prove (iii), we note that by Theorem \ref{thm_i3}, 

$$i_3(G) = {n \choose 3} - n(n-2) + \sum_{i=2}^{n-1} {i \choose 2}g_i-n(C_3)$$.

Furthermore $i_3(C_{n})$ can easily be computed to be ${n \choose 3} - n(n-2) + n$. As $i_3(G)=i_3(C_{n})$ it follows that (iii) holds. Finally by adding (i) and (iii) and subtracting (ii) we obtain:

$$n(C_3)=\sum_{i=0}^{n-1} g_i+ \sum_{i=2}^{n-1} {i \choose 2}g_i - \sum_{i=1}^{n-1} i \cdot g_i = g_0 + \sum_{i=3}^{n-1} \left({i \choose 2} -i+1 \right) g_i \geq g_0+\sum_{i=3}^{n-1} g_i.$$

Hence (iv) holds as well.
\end{proof}

We will require basic computational results on computing the independence polynomial due to Gutman and Harary.

\begin{prop}[\cite{INDFIRST}]\label{prop:deletion}
Let $G$ and $H$ be graphs and $v\in V(G)$. Then:
\begin{enumerate}[i)]
\item $i(G,x)=i(G-v,x)+xi(G-N[v],x)$.
\item $i(G\cup H,x)=i(G,x)\cdot i(H,x)$. \QED
\end{enumerate}
\end{prop}

\subsection{Even Cycles}

\begin{theorem}
\label{thm:evencycle} 

Let $K_4-e$ denote the graph which consists of a $K_4$ with one edge removed. Then

\begin{itemize}
%\item $[C_4]=\{C_4,D_4\}$,
\item $[C_{6}] = \{C_{2n},D_{2n},(K_4-e) \cup K_2\},$ and
\item $[C_{2n}] = \{C_{2n},D_{2n}\}$ for $n \geq 2,~n \neq 3$.

\end{itemize}

\end{theorem}

\begin{proof}
Suppose $G \sim C_{2n}$ and $G \not\cong C_{2n}$. Then $G$ has $2n$ vertices and $2n$ edges. For $n=2$ there is only one graph, $D_4$, with $4$ edges and $4$ vertices which is not isomorphic to $C_4$. As  $C_4 \sim D_4$ by Proposition \ref{prop:chismcycle} then $[C_{4}] = \{C_{2n},D_{2n}\}$. We now consider when $n \geq 3$. By Theorem \ref{thm:PathPoly} and Proposition \ref{prop:deletion} it can be shown that $i(G,x)$ is degree $n$ with leading coefficient equal to $2$. That is, there are exactly two maximum independent sets in $G$ of size $n$.

We begin by showing $G$ contains a triangle. Suppose not, and let $g_i$ be the number of vertices of degree $i$ in $G$. By Lemma \ref{lem:eq3} (iii) and (iv), as $G$ is triangle-free (i.e. $n(C_3)=0$) and $G \sim C_{2n}$, 

$$  \sum_{i=2}^{2n-1} {i \choose 2}g_i =2n \text{ } \text{  and   } \text{ } 0 \geq g_0 + \sum_{i=3}^{2n-1} g_i.$$

\noindent Hence $g_i = 0$ for $i \geq 3$ and thus  $\sum_{i=2}^{2n-1} {i \choose 2}g_i =2n$ implies $G$ is 2-regular. However as $G \not\cong C_{2n}$ then $G$ is a disjoint union of cycles. It is easy to see each cycle has at least two maximum independent sets, meaning $G$ must have at least $4$ maximum independent sets which is a contradiction. Thus $G$ contains a triangle.

As $G$ contains a triangle, it is not bipartite, and hence the two maximum independent sets (of cardinality $n$) in $G$ are not disjoint. Thus we can partition the vertices into non-empty sets $A, A', B, B'$ such that $A \cup A'$ and $A \cup B$ are the two independent sets of size $n$. Note $|A \cup A'|=|A \cup B|=|B \cup B'|=n$ and $|A'|=|B|$. It follows that $|A| = |B'|$ and so $|A'|+|B'|=n$.

Each vertex in $B'$ is adjacent to at least two vertices in $A \cup A'$. Otherwise we can form another independent set of size at least $n$ which is not $A \cup A'$ nor $A \cup B$. Thus our partially constructed $G$ looks like:

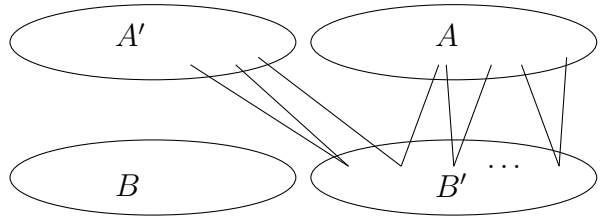
\begin{center}
\scalebox{1}{
\begin{tikzpicture}
\draw (0,0.9) ellipse (1.9 and 0.5);
\node[text width=1cm] at (0,1) {\large $A'$};

\draw (4,0.9) ellipse (1.9 and 0.5);
\node[text width=1cm] at (4.25,1) {\large $A$};

\draw (0,-0.9) ellipse (1.9 and 0.5);
\node[text width=1cm] at (0,-1) {\large $B$};

\draw (4,-0.9) ellipse (1.9 and 0.5);
\node[text width=1cm] at (4.25,-1) {\large $B'$};

%\node[shape=circle,draw=black,fill=black] (1) at (-1.4,-0.75) {};
%\node[shape=circle,draw=black,fill=black] (2) at (-0.7,-0.75) {};
%\node[shape=circle,draw=black,fill=black] (3) at (0,-0.75) {};
%\node[shape=circle,draw=black,fill=black] (4) at (1.4,-0.75) {};

\begin{scope}

    \path [-] (-1.4+4,-0.75) edge node {} (-1.5+2,0.6);
    \path [-] (-1.4+4,-0.75) edge node {} (-0.9+2,0.6);
    
    \path [-] (-0.7+4,-0.75) edge node {} (-0.8+2.2,0.7);
    \path [-] (-0.7+4,-0.75) edge node {} (-0.2+4,0.6);
    
    \path [-] (0+4,-0.75) edge node {} (-0.1+4,0.6);
    \path [-] (0+4,-0.75) edge node {} (0.5+4,0.6);
    
    \path [-] (1.4+4,-0.75) edge node {} (0.9+4,0.6);
    \path [-] (1.4+4,-0.75) edge node {} (1.5+4,0.7);
    
    \path (0+4,-0.75) -- node[auto=false]{\ldots} (1.4+4,-0.75);

\end{scope}

%\draw[black, decoration={brace, raise=5pt, mirror, amplitude=3mm}, decorate] (-1.5,-1) -- (1.5,-1);

%\node[text width=1cm] at (0.4,-1.75) {$n$};
\end{tikzpicture}}
\end{center}

We now consider two cases: $|B| \geq 2$ and $|B|=1$. If $|B| \geq 2$, then by the same argument used for  $B'$ and $A \cup A'$, each vertex in $A'$ is adjacent to at least two vertices in $B$. Thus $G$ has is at least $2(|A'|+|B'|)=2n$ edges. However, as $G$ is not bipartite and has exactly $2n$ edges, there must be an edge between two vertices of $B \cup B'$, a contradiction.

Now suppose $|B|=1$. As $|B|=|A'|$, we now have that $|A'|=1$. We will label the vertex in $A'$ and the vertex in $B$ to be $a'$ and $b$, respectively. Note $a'$ and $b$ are adjacent, as otherwise $A \cup A' \cup B$ forms a independent set of size $n+1$. Thus our partially constructed $G$ (omitting one edge in $B\cup B'$) looks like this:
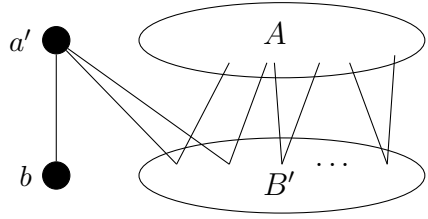
\begin{center}
\scalebox{1}{
\begin{tikzpicture}
%\draw (0,0.9) ellipse (1.9 and 0.5);
%\node[text width=1cm] at (1-0.5,0.9) {\large $A'$};

\draw (4,0.9) ellipse (1.9 and 0.5);
\node[text width=1cm] at (4.25,1) {\large $A$};

%\draw (0,-0.9) ellipse (1.9 and 0.5);
%\node[text width=1cm] at (1-0.5,-0.9) {\large $B$};

\draw (4,-0.9) ellipse (1.9 and 0.5);
\node[text width=1cm] at (4.25,-1) {\large $B'$};

\node[shape=circle,draw=black,fill=black,label=left:$a'$] (a) at (1,0.9) {};
\node[shape=circle,draw=black,fill=black,label=left:$b$] (b) at (1,-0.9) {};
%\node[shape=circle,draw=black,fill=black] (3) at (0,-0.75) {};
%\node[shape=circle,draw=black,fill=black] (4) at (1.4,-0.75) {};

\begin{scope}

    \path [-] (-1.4+4,-0.75) edge node {} (a);
    \path [-] (-1.4+4,-0.75) edge node {} (-0.7+4,0.6);
    
    \path [-] (-0.7+4,-0.75) edge node {} (a);
    \path [-] (-0.7+4,-0.75) edge node {} (-0.2+4,0.6);
    
    \path [-] (0+4,-0.75) edge node {} (-0.1+4,0.6);
    \path [-] (0+4,-0.75) edge node {} (0.5+4,0.6);
    
    \path [-] (1.4+4,-0.75) edge node {} (0.9+4,0.6);
    \path [-] (1.4+4,-0.75) edge node {} (1.5+4,0.7);
    
    \path [-] (a) edge node {} (b);
    
    \path (0+4,-0.75) -- node[auto=false]{\ldots} (1.4+4,-0.75);

\end{scope}

%\draw[black, decoration={brace, raise=5pt, mirror, amplitude=3mm}, decorate] (-1.5,-1) -- (1.5,-1);

%\node[text width=1cm] at (0.4,-1.75) {$n$};
\end{tikzpicture}}
\end{center}

We will consider the placement of the edge in $B \cup B'$ in two cases.

\vspace{3mm}
\noindent \textbf{Case 1: The edge is from $b$ to some vertex $v \in B'$,}
\vspace{3mm}

Then as $G$ contains a triangle, $v$ must be adjacent to $a'$ (note this is only triangle in $G$). All vertices in $B \cup B'$ are now degree two with the exception of $v$ which has degree three. Thus $G$ now looks like:

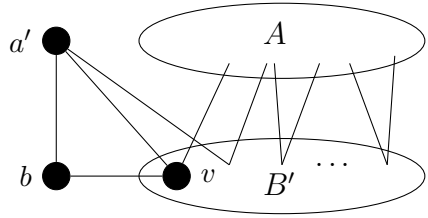
\begin{center}
\scalebox{1}{
\begin{tikzpicture}
%\draw (0,0.9) ellipse (1.9 and 0.5);
%\node[text width=1cm] at (1-0.5,0.9) {\large $A'$};

\draw (4,0.9) ellipse (1.9 and 0.5);
\node[text width=1cm] at (4.25,1) {\large $A$};

%\draw (0,-0.9) ellipse (1.9 and 0.5);
%\node[text width=1cm] at (1-0.5,-0.9) {\large $B$};

\draw (4,-0.9) ellipse (1.9 and 0.5);
\node[text width=1cm] at (4.25,-1) {\large $B'$};

\node[shape=circle,draw=black,fill=black,label=left:$a'$] (a) at (1,0.9) {};
\node[shape=circle,draw=black,fill=black,label=left:$b$] (b) at (1,-0.9) {};
\node[shape=circle,draw=black,fill=black,label=right:$v$] (v) at (-1.4+4,-0.9) {};
%\node[shape=circle,draw=black,fill=black] (3) at (0,-0.75) {};
%\node[shape=circle,draw=black,fill=black] (4) at (1.4,-0.75) {};

\begin{scope}

    \path [-] (v) edge node {} (a);
    \path [-] (v) edge node {} (b);
    \path [-] (v) edge node {} (-0.7+4,0.6);
    
    \path [-] (-0.7+4,-0.75) edge node {} (a);
    \path [-] (-0.7+4,-0.75) edge node {} (-0.2+4,0.6);
    
    \path [-] (0+4,-0.75) edge node {} (-0.1+4,0.6);
    \path [-] (0+4,-0.75) edge node {} (0.5+4,0.6);
    
    \path [-] (1.4+4,-0.75) edge node {} (0.9+4,0.6);
    \path [-] (1.4+4,-0.75) edge node {} (1.5+4,0.7);
    
    \path [-] (a) edge node {} (b);
    
    \path (0+4,-0.75) -- node[auto=false]{\ldots} (1.4+4,-0.75);

\end{scope}

%\draw[black, decoration={brace, raise=5pt, mirror, amplitude=3mm}, decorate] (-1.5,-1) -- (1.5,-1);

%\node[text width=1cm] at (0.4,-1.75) {$n$};
\end{tikzpicture}}
\end{center}

We now know that $G$ has exactly one triangle (i.e. $n(C_3)=1$) and $G \sim C_{2n}$ so, by Lemma \ref{lem:eq3} (iv), $G$ has at most one vertex which is not degree one or two. As $v$ is degree three then every other vertex must either have degree one or two. Again let $g_i$ be the number of vertices of degree $i$ in $G$. Note $g_i=0$ for $i \neq 1,2,3$, $g_3=1$, and $g_1+g_2+g_3=2n$. Furthermore by Lemma \ref{lem:eq3} (iii),

$$2n+1=\sum_{i=2}^{2n-1} {i \choose 2}g_i=g_2+3.$$

\noindent Thus $g_2=2n-2$, $g_3=1$ and $g_1=1$. Note $a'$ must have degree two. We now construct $G$. Begin with the one triangle in $G$ which is formed by the vertices $a'$, $b$, and $v$. As $v$ is a degree three vertex it must have a neighbour in $A$. Now label the only vertex of degree one as $\ell$. As the other vertices in $G$ are all degree two, there must be an induced path of vertices connecting $v$ and $\ell$. This forms a $D_r$ component in $G$ for some $r \leq n$. If $r=n$ then $G \cong D_n$, otherwise $G$ is the disjoint union of cycles and a $D_r$ for $r < n$. However as $D_n$ has two maximum independent sets, if $G$ has any cycle components it would have at least four maximum independent sets which is a contradiction.   

\vspace{3mm}
\noindent \textbf{Case 2:}
\vspace{3mm}

If the edge in $B\cup B'$ is between two vertices $u,v \in B'$, then as $G$ contains a triangle, $u$ and $v$ must have at least one common neighbour in $A \cup A'$. Note that we now know the number of vertices of each degree in $B \cup B'$; $b$ is degree one, $u$ and $v$ are degree three and every other vertex in $B \cup B'$ is degree two. Thus we consider two subcases: $u$ and $v$ have one or two common neighbours.

\vspace{3mm}
\noindent \textbf{Case 2a: $u$ and $v$ have exactly one common neighbour}
\vspace{3mm}

Then $G$ has exactly one triangle and now looks like:

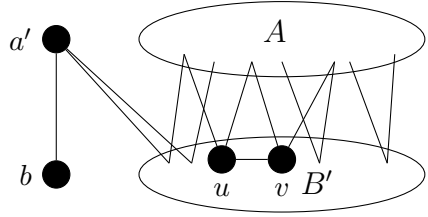
\begin{center}
\scalebox{1}{
\begin{tikzpicture}
%\draw (0,0.9) ellipse (1.9 and 0.5);
%\node[text width=1cm] at (1-0.5,0.9) {\large $A'$};

\draw (4,0.9) ellipse (1.9 and 0.5);
\node[text width=1cm] at (4.25,1) {\large $A$};

%\draw (0,-0.9) ellipse (1.9 and 0.5);
%\node[text width=1cm] at (1-0.5,-0.9) {\large $B$};

\draw (4,-0.9) ellipse (1.9 and 0.5);
\node[text width=1cm] at (4.75,-1) {\large $B'$};

\node[shape=circle,draw=black,fill=black,label=left:$a'$] (a) at (1,0.9) {};
\node[shape=circle,draw=black,fill=black,label=left:$b$] (b) at (1,-0.9) {};

\node[shape=circle,draw=black,fill=black,label=below:$u$] (u) at (3.2,-0.7) {};
\node[shape=circle,draw=black,fill=black,label=below:$v$] (v) at (4,-0.7) {};

%\node[shape=circle,draw=black,fill=black] (3) at (0,-0.75) {};
%\node[shape=circle,draw=black,fill=black] (4) at (1.4,-0.75) {};

\begin{scope}

    \path [-] (-1.5+4,-0.75) edge node {} (a);
    \path [-] (-1.5+4,-0.75) edge node {} (-1.3+4,0.7);
    
    \path [-] (-1.2+4,-0.75) edge node {} (a);
    \path [-] (-1.2+4,-0.75) edge node {} (-0.9+4,0.6);
    
    \path [-] (0.5+4,-0.75) edge node {} (0+4,0.6);
    \path [-] (0.5+4,-0.75) edge node {} (0.7+4,0.6);
    
    \path [-] (1.4+4,-0.75) edge node {} (0.9+4,0.6);
    \path [-] (1.4+4,-0.75) edge node {} (1.5+4,0.7);
    
    \path [-] (a) edge node {} (b);
    \path [-] (u) edge node {} (v);
    \path [-] (u) edge node {} (3.6,.6);
    \path [-] (u) edge node {} (-1.3+4,0.7);
    \path [-] (v) edge node {} (3.6,.6);
    \path [-] (v) edge node {} (0.7+4,0.6);
    
    %\path (0+4,-0.75) -- node[auto=false]{\ldots} (1.4+4,-0.75);

\end{scope}

%\draw[black, decoration={brace, raise=5pt, mirror, amplitude=3mm}, decorate] (-1.5,-1) -- (1.5,-1);

%\node[text width=1cm] at (0.4,-1.75) {$n$};
\end{tikzpicture}}
\end{center}

As $G$ has exactly one triangle (i.e. $n(C_3)=1$) and $G \sim C_{2n}$, Lemma \ref{lem:eq3} (iv) gives that $g_3\le 1$. However $u$ and $v$ both have degree three which is a contradiction.

\vspace{3mm}
\noindent \textbf{Case 2b: $u$ and $v$ have exactly two common neighbours.}
\vspace{3mm}

Then $G$ has exactly two triangles and looks like:

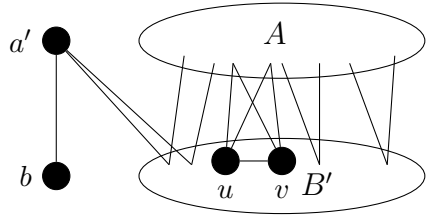
\begin{center}
\scalebox{1}{
\begin{tikzpicture}
%\draw (0,0.9) ellipse (1.9 and 0.5);
%\node[text width=1cm] at (1-0.5,0.9) {\large $A'$};

\draw (4,0.9) ellipse (1.9 and 0.5);
\node[text width=1cm] at (4.25,1) {\large $A$};

%\draw (0,-0.9) ellipse (1.9 and 0.5);
%\node[text width=1cm] at (1-0.5,-0.9) {\large $B$};

\draw (4,-0.9) ellipse (1.9 and 0.5);
\node[text width=1cm] at (4.75,-1) {\large $B'$};

\node[shape=circle,draw=black,fill=black,label=left:$a'$] (a) at (1,0.9) {};
\node[shape=circle,draw=black,fill=black,label=left:$b$] (b) at (1,-0.9) {};

\node[shape=circle,draw=black,fill=black,label=below:$u$] (u) at (3.25,-0.7) {};
\node[shape=circle,draw=black,fill=black,label=below:$v$] (v) at (4,-0.7) {};

%\node[shape=circle,draw=black,fill=black] (3) at (0,-0.75) {};
%\node[shape=circle,draw=black,fill=black] (4) at (1.4,-0.75) {};

\begin{scope}

    \path [-] (-1.5+4,-0.75) edge node {} (a);
    \path [-] (-1.5+4,-0.75) edge node {} (-1.3+4,0.7);
    
    \path [-] (-1.2+4,-0.75) edge node {} (a);
    \path [-] (-1.2+4,-0.75) edge node {} (-0.9+4,0.6);
    
    \path [-] (0.5+4,-0.75) edge node {} (0+4,0.6);
    \path [-] (0.5+4,-0.75) edge node {} (0.5+4,0.6);
    
    \path [-] (1.4+4,-0.75) edge node {} (0.9+4,0.6);
    \path [-] (1.4+4,-0.75) edge node {} (1.5+4,0.7);
    
    \path [-] (a) edge node {} (b);
    \path [-] (u) edge node {} (v);
    \path [-] (u) edge node {} (3.85,.6);
    \path [-] (u) edge node {} (3.35,.6);
    \path [-] (v) edge node {} (3.85,.6);
    \path [-] (v) edge node {} (3.35,0.6);
    
    %\path (0+4,-0.75) -- node[auto=false]{\ldots} (1.4+4,-0.75);

\end{scope}

%\draw[black, decoration={brace, raise=5pt, mirror, amplitude=3mm}, decorate] (-1.5,-1) -- (1.5,-1);

%\node[text width=1cm] at (0.4,-1.75) {$n$};
\end{tikzpicture}}
\end{center}

Since $G$ has exactly two triangles (i.e. $n(C_3)=2$) and $G \sim C_{2n}$, Lemma \ref{lem:eq3} (iv) implies that $\displaystyle{\sum_{i\neq 1,2}}g_i\le 2$. Both $u$ and $v$ have degree three so every other vertex must either have degree one or two. Note $g_i=0$ for $i \neq 1,2,3$, $g_3=2$, and $g_1+g_2+g_3=2n$. Furthermore by Lemma \ref{lem:eq3} (iii),

$$2n+2=\sum_{i=2}^{2n-1} {i \choose 2}g_i=g_2+6.$$

\noindent Thus $g_2=2n-4$, $g_3=2$ and $g_1=2$. Note that every in $A \cup A'$ has degree at most three, thus $u$, $v$ and their two common neighbours form a $K_4$ less an edge component of $G$. Furthermore $G$ has two vertices of degree one. As $b$ is degree one and every vertex in $B'$ is degree two or three, the second of vertex degree one is $a'$ or some vertex in $A$.

First suppose some vertex $\ell \in A$ is degree one. At this point our graph looks like:

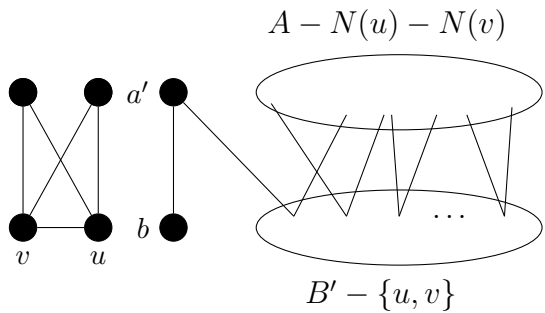
\begin{center}
\scalebox{1}{
\begin{tikzpicture}
%\draw (0,0.9) ellipse (1.9 and 0.5);
%\node[text width=1cm] at (1-0.5,0.9) {\large $A'$};

\draw (4,0.9) ellipse (1.9 and 0.5);
\node[text width=5cm] at (4.75,1.8) {\large $A-N(u)-N(v)$};

%\draw (0,-0.9) ellipse (1.9 and 0.5);
%\node[text width=1cm] at (1-0.5,-0.9) {\large $B$};

\draw (4,-0.9) ellipse (1.9 and 0.5);
\node[text width=5cm] at (5.25,-1.8) {\large $B'-\{u,v\}$};

\node[shape=circle,draw=black,fill=black,label=left:$a'$] (a) at (1,0.9) {};
\node[shape=circle,draw=black,fill=black,label=left:$b$] (b) at (1,-0.9) {};

\node[shape=circle,draw=black,fill=black,label=below:$u$] (u) at (0,-0.9) {};
\node[shape=circle,draw=black,fill=black,label=below:$v$] (v) at (-1,-0.9) {};
\node[shape=circle,draw=black,fill=black] (1) at (0,0.9) {};
\node[shape=circle,draw=black,fill=black] (2) at (-1,0.9) {};

\node[shape=circle,draw=black] (1) at (0,0.9) {};
\node[shape=circle,draw=black] (2) at (-1,0.9) {};
%\node[shape=circle,draw=black,fill=black] (3) at (0,-0.75) {};
%\node[shape=circle,draw=black,fill=black] (4) at (1.4,-0.75) {};

\begin{scope}

    \path [-] (-1.4+4,-0.75) edge node {} (a);
    \path [-] (-1.4+4,-0.75) edge node {} (-0.7+4,0.6);
    
    \path [-] (-0.7+4,-0.75) edge node {} (-0.7+3,0.75);
    \path [-] (-0.7+4,-0.75) edge node {} (-0.2+4,0.6);
    
    \path [-] (0+4,-0.75) edge node {} (-0.1+4,0.6);
    \path [-] (0+4,-0.75) edge node {} (0.5+4,0.6);
    
    \path [-] (1.4+4,-0.75) edge node {} (0.9+4,0.6);
    \path [-] (1.4+4,-0.75) edge node {} (1.5+4,0.7);
    
    \path [-] (a) edge node {} (b);
    \path [-] (u) edge node {} (v);
    \path [-] (u) edge node {} (1);
    \path [-] (u) edge node {} (2);
    \path [-] (v) edge node {} (1);
    \path [-] (v) edge node {} (2);
    
    \path (0+4,-0.75) -- node[auto=false]{\ldots} (1.4+4,-0.75);

\end{scope}

%\draw[black, decoration={brace, raise=5pt, mirror, amplitude=3mm}, decorate] (-1.5,-1) -- (1.5,-1);

%\node[text width=1cm] at (0.4,-1.75) {$n$};
\end{tikzpicture}}
\end{center}

Note that every vertex in $B'-\{u,v\}$ and $A-N(u)-N(v)$ is degree two other than $\ell$. Therefore one component in $G$ is a even order path from $b$ to $\ell$. However, every even path with more than two vertices has at least three maximum independent sets, which is a contradiction as $G$ only has two maximum independent sets.

Now suppose $a'$ is degree one. Then $a'$ and $b$ form a $K_2$ component in $G$ and the remaining vertices in $(B'-\{u,v\})\cup(A-N(u))$ must induce a disjoint union of cycles. In the case where $n=3$, that is $G \sim C_6$, $G$ has no cycle components and $G \cong (K_4-e) \cup K_2$. For $n \geq 4$, $(B'-\{u,v\})\cup(A-N(u))$ contains at least one cycle. However, as $K_2$ and cycles each have at least two maximum independent sets, $G$ has at least four maximum independent sets, which is again a contradiction.

The only two cases which didn't result in a contradiction yielded $G \cong D_{2n}$ and $G \cong (K_4-e) \cup K_2$. As $D_{2n} \sim C_{2n}$ for all $n \geq 3$ and $(K_4-e) \cup K_2 \sim C_6$ we have shown that  $[C_{6}] = \{C_{6},D_{6}, (K_4-e) \cup K_2 \}$ and $[C_{2n}] = \{D_{2n},D_{2n}\}$ for $n \geq 4$.
\end{proof}
\subsection{Prime Power Cycles}
In Theorem \ref{thm:evencycle}, we used an involved construction to show that there is only one disconnected graph that is independence equivalent to $C_{2n}$. This construction relies on the fact that the leading coefficient of $i(C_{2n},x)$ is 2. This argument will not hold for odd cycles as the leading coefficient of $i(C_{2n+1},x)$ is $2n+1$. However there are other ways to show connectivity, and we shall do so via irreducibility of polynomials over the rationals. We will Eisenstein's famous criterion for irreducibility that we state here

\begin{theorem}[c.f. \cite{Fraleigh} pp. 215]
Let $p\in \mathbb{Z}$ be a prime and $f(x)=a_0+a_1x+\ldots+a_nx^n$ be a polynomial of degree $n$ with integer coefficients. If $p$ divides each of $a_0,a_1,\ldots,a_{n-1}$ but $p$ does not divide $a_n$, and $p^2$ does not divide $a_0$, then $f$ is irreducible over the rationals. \QED
\end{theorem}

\begin{prop}\label{prop:primecycles}
If $p$ is an odd prime, then $[C_p]=\{C_p,D_p\}$ (note $C_p \cong D_p$ when $p=3$).
\end{prop}
\begin{proof}
We show that $i(C_p,x)$ is irreducible over the rationals and therefore $C_p$ has no disconnected graphs in its equivalence class. The result will then follow by Proposition~\ref{thm:cycleconnectedequivclass}. Let $p$ be an odd prime. By Theorem~\ref{thm:PathPoly} and Proposition~\ref{prop:deletion} we know that 
\begin{align*}
i(C_p,x)&=i(P_{p-1},x)+xi(P_{p-3},x)\\
&= \sum_{j=0}^{\lfloor \frac{p}{2} \rfloor} { {p -j} \choose {j}}x^j+\sum_{j=0}^{\lfloor \frac{p-2}{2} \rfloor} { {p-2-j} \choose {j}}x^{j+1}\\
&=\sum_{j=0}^{\lfloor \frac{p}{2} \rfloor} { {p -j} \choose {j}}x^j+\sum_{j=1}^{\lfloor \frac{p}{2} \rfloor} { {p-j-1} \choose {j-1}}x^{j}\\
&=1+\sum_{j=1}^{\lfloor \frac{p}{2} \rfloor}\left( { {p -j} \choose {j}}+{ {p-j-1} \choose {j-1}}\right)x^j\\
&=1+\sum_{j=1}^{\lfloor \frac{p}{2} \rfloor}{ {p -j} \choose {j}}\left(\frac{p}{p-j}\right)x^j.\\
\end{align*}

The coefficients above must be integers and since $p$ is a prime, it follows that $p-j$ does not divide $p$ for any $j=1,2,\ldots, \lfloor \frac{p}{2} \rfloor$, so $p-j$ must divide the integer ${ {p -j} \choose {j}}$.  Therefore, ${ {p -j} \choose {j}}\left(\frac{p}{p-j}\right)$ is a multiple of $p$ for $j=1,2,\ldots, \lfloor \frac{p}{2} \rfloor$.  We now consider the coefficient of $x^{\lfloor \frac{p}{2} \rfloor}$,
\begin{align*}
{ {p -\lfloor \frac{p}{2} \rfloor} \choose {\lfloor \frac{p}{2} \rfloor}}\left(\frac{p}{p-\lfloor \frac{p}{2} \rfloor}\right)&=\left(\frac{(p -\lfloor \frac{p}{2}\rfloor-1)!}{\lfloor \frac{p}{2}\rfloor!(p -2\lfloor \frac{p}{2}\rfloor)!} \right)p\\
&=\left(\frac{(p -\lceil \frac{p}{2}\rceil)!}{\lfloor \frac{p}{2}\rfloor!(\lceil \frac{p}{2}\rceil) -\lfloor \frac{p}{2}\rfloor)!} \right)p\\
&=\left(\frac{\lfloor \frac{p}{2}\rfloor!}{\lfloor \frac{p}{2}\rfloor!} \right)p\\
&=p.
\end{align*}

Therefore, applying Eisenstein's famous criterion to the polynomial $x^{\alpha(C_p)}i(C_p,\tfrac{1}{x})$ with the prime $p$, it follows that $i(C_p,x)$ is irreducible over the rationals. Since $i(C_p,x)$ is irreducible, $C_p$ cannot be independence equivalent to any disconnected graph. It follows that $[C_p]=\{C_p,D_p\}$ by Theorem~\ref{thm:cycleconnectedequivclass}.

\end{proof}
%\subsection{Prime Power Cycles}
The irreducibility of cycles of prime length given in Proposition~\ref{prop:primecycles} can be partially extended to cycle with length $p^n$ for all $n$ and all odd primes $p\ge 5$. These polynomials are reducible but considering each irreducible factor will lead us to the same conclusion as the case for $n=1$.

\begin{defn}\label{def:cyclicpoly}
We say that a polynomial $p(x)=\sum_{i=0}^np_ix^i$ with integer coefficients is \textbf{unicyclic} if $p_0=1$, $p_1=k$ and $p_2=\binom{k}{2}-k$ for some integer $k$.
\end{defn}

Note that a unicyclic polynomial is one that shares the same first three coefficients with the independence polynomial of some unicyclic graph. If a connected graph has a unicyclic independence polynomial, then that graph must be unicyclic. This is because the graph has $n$ vertices, $n$ edges, and is connected.

\begin{lemma}\label{lem:productunicyclic}
If $h(x)=g(x)f(x)$ and $h(x)$, $g(x)$ are unicyclic, then $f(x)$ is unicyclic.
\end{lemma}

\begin{proof}
Assuming the hypothesis, let the first three terms of $g(x)$ be $1,nx,\left(\binom{n}{2}-n\right)x^2$, the first three terms of  $f(x)$ be $1,kx,\left(\binom{k}{2}-k+\ell\right)x^2$ where $\ell$ is some integer, so that the first three terms of $h(x)$ are $1,(n+k)x,\left(\binom{n+k}{2}-(n+k)\right)x^2$. Since $h(x)=f(x)g(x)$, they must be equal coefficient-wise so we must have,

\begin{align*}
\binom{n+k}{2}-(n+k)&=\binom{n}{2}-n+\binom{k}{2}-k+\ell+nk\\
&=\frac{n(n-1)+k(k-1)+2nk}{2}-(n+k)+\ell\\
&=\frac{(n+k)((n+k)-1)}{2}-(n+k)+\ell\\
&=\frac{(n+k)^2-(n+k)}{2}-(n+k)+\ell\\
&=\binom{n+k}{2}-(n+k)+\ell.\\
\end{align*}

Therefore, $\ell=0$, and $f(x)$ is unicyclic.
\end{proof}

The roots of $i(C_n,x)$ have been completely determined by Alikahni and Peng \cite{AlikhaniPeng2011} and we will make use of a corollary that can be derived from their results.

\begin{theorem}[\cite{AlikhaniPeng2011}]\label{thm:cycleroots}
For $n\ge 3$, the roots of $i(C_n,x)$ are given by

$$r_i=-\frac{1}{2\left(1+\cos\left(\frac{(2i-1)\pi}{n}\right) \right)}$$
for $i=1,2,\ldots,\lfloor\frac{n}{2}\rfloor$, and these roots are all distinct.
\QED
\end{theorem}

\begin{cor}\label{cor:cycledivision}
For odd $n$ and $k\neq 1$, $k|n$ if and only if $i(C_k,x)|i(C_n,x)$.
\end{cor}
\begin{proof}
Let $n$ be odd. First suppose $k|n$. Then let $n=qk$ for some positive integer $q$. By Theorem~\ref{thm:cycleroots}, we only have to show that for all $j=1,2,\ldots,\lfloor\frac{k}{2}\rfloor$ there exists an $i$ from $1\le i\le \lfloor\frac{n}{2}\rfloor$ such that $\frac{(2i-1)\pi}{n}=\frac{(2j-1)\pi}{k}$. This happens if and only if 
$$i=\frac{(2j-1)q+1}{2}.$$

Since $n$ is odd, it follows that $q$ is also odd and therefore $i$ is indeed an integer and since $j\le \lfloor\frac{k}{2}\rfloor$, $1\le i\le \lfloor\frac{n}{2}\rfloor$. Thus every root of $i(C_k,x)$ is also a root of $i(C_n,x)$. Let $i(C_k,x)=(x-r_1)(x-r_2)\ldots(x-r_{\lfloor\frac{k}{2}\rfloor})$ where the $r_i$'s are the roots of $i(C_k,x)$. Since all roots of $i(C_k,x)$ are also roots of $i(C_n,x)$, it follows that $i(C_n,x)=(x-r_1)(x-r_2)\ldots(x-r_{\lfloor\frac{k}{2}\rfloor})g(x)$ for some polynomial $g(x)$ and therefore $i(C_k,x)|i(C_n,x)$.

Conversely suppose $i(C_k,x)|i(C_n,x)$. Then the leading coefficient of $i(C_k,x)$ must divide the leading coefficient of $i(C_n,x)$. From Theorem \ref{thm:PathPoly} and Proposition \ref{prop:deletion}, as $n$ is odd then the leading coefficient of $i(C_n,x)$ is $n$. Furthermore the leading coefficient of $i(C_k,x)$ is either 2 if $k$ is even or $k$ if $k$ is odd. As $n$ is odd then $2 \not | n$ and hence $k|n$.
\end{proof}

\begin{lemma}
\label{lem:cycleuni}
Let $p$ be an odd prime and $n\ge 1$. Then every irreducible factor of $i(C_{p^n},x)$ is unicyclic. 
\end{lemma} 

\begin{proof}
The proof is by induction on $n$. For $n=1$, that case was handled in Proposition~\ref{prop:primecycles}. Suppose the result holds for $n\le k$ for some $k\ge 1$. Now from Corollary~\ref{cor:cycledivision}, we know that $i(C_{p^{k}},x)|i(C_{p^{k+1}},x)$. Let $i(C_{p^{k+1}},x)=i(C_{p^k},x)r(x)$. We claim that $r(x)$ is irreducible and unicyclic. The fact that $r(x)$ is unicyclic follows from the inductive hypothesis and Lemma~\ref{lem:productunicyclic}, once we show that $r(x)$ is irreducible.

Similarly to the proof of Proposition~\ref{prop:primecycles}, we derive and expression for the coefficients of $i(C_{p^k},x)$

$$i(C_{p^k},x)=1+\sum_{j=1}^{\lfloor \frac{p^k}{2} \rfloor}{ {p^k -j} \choose {j}}\left(\frac{p^k}{p^k-j}\right)x^j.
$$

Note that $p$ divides each coefficient above except the constant term as $\frac{p^k}{p^k-j}$ must be an integer and $p^k-j$ has at most $k-1$ factors of $p$ for all $1\le j\le p^k-1$.

Let $r(x)=r_0+r_1x+r_2x^2+\cdots+r_{m}x^m$. 

Since $i(C_{p^{k+1}},x)=i(C_{p^k},x)r(x)$, we must have

\begin{align}
{ {p^{k+1} -j} \choose {j}}\left(\frac{p^{k+1}}{p^{k+1}-j}\right)&=\sum_{i=0}^{j}\left(r_i{ {p^k -(j-i)} \choose {j-i}}\left(\frac{p^k}{p^k-(j-i)}\right)\right)\label{eq:primepwrsum}
\end{align}
for $j=0,1,\ldots,\lfloor \frac{p^{k+1}}{2} \rfloor$.

As noted earlier, since $p\vert\frac{p^{k+1}}{p^{k+1}-j}$ for $1\le j\le p^{k+1}-1$, $p$ must divide the sum on the right hand side of (\ref{eq:primepwrsum}). Since we know $p$ divides each coefficient of $i(C_{p^k},x)$ except the constant term, it follows that $p|r_j$ for all $j=1,2,\ldots,m$.  Also, since $p^kr_m=p^{k+1}$, it follows that $r_m=p$. So by Eisenstein's Criterion applied to $x^mr(\tfrac{1}{x})$, it follows that $r(x)$ is irreducible.
\end{proof}

\begin{theorem}
\label{thm:powerprimecycles}
For $k,p \in \mathbb{N}$ where $p \geq 5$ is prime, $[C_{p^k}] = \{C_{p^k},D_{p^k}\}$.
\end{theorem}

\begin{proof}
Suppose $G \sim C_{n}$ and $G \not\cong C_{n}$ where $n=p^k$. Then $G$ has $n$ vertices and $n$ edges. Then by Lemma \ref{lem:eq3} we obtain the following three equations:

$$\sum_{i=0}^{n-1} g_i = n, \text{     }\sum_{i=1}^{n-1} i \cdot g_i = 2n, \text{     } \sum_{i=2}^{n-1} {i \choose 2}g_i =n+n(C_3).$$

\noindent Thus, 

\[n(C_3)=\sum_{i=0}^{n-1} g_i+ \sum_{i=2}^{n-1} {i \choose 2}g_i - \sum_{i=1}^{n-1} i \cdot g_i = g_0 + \sum_{i=3}^{n-1} \left({i \choose 2} -i+1 \right) g_i. \tag{1} \]

\noindent Furthermore, $G$ has no $C_3$ components, otherwise $i(C_3) | i(C_n)$ and  hence by Corollary \ref{cor:cycledivision}, $3|n$ which is a contradiction as $n=p^k$ for prime $p \geq 5$. Hence every induced $C_3$ has a vertex with degree 3 or greater. By Lemma \ref{lem:cycleuni}, every irreducible factor of $i(C_n)$ is unicyclic and hence every connected component of $G$ has the same number of vertices and edges and is therefore unicyclic. Therefore every vertex is part of at most one induced $C_3$. As every induced $C_3$ has a vertex with degree 3 or greater then

\[n(C_3) \leq \sum_{i=3}^{n-1} g_i.\]

\noindent Therefore by subtracting this inequality from equation $(1)$ we obtain

$$0 \geq g_0 + \sum_{i=3}^{n-1} \left({i \choose 2} -i \right) g_i.$$

\noindent As ${i \choose 2} -i \geq 2$ for $i \geq 4$ then $g_i=0$ for $i \neq 1,2 \text{ or }3$. Therefore, by equation (1) we have $g_3=n(C_3)$. We can also now simplify the sums given in Lemma~\ref{lem:eq3} to get $g_1+g_2+g_3=n$ and $g_1+2g_2+3g_3=2n$ and subtracting $2$ times the former from the latter we obtain $g_1=g_3$.

Consider the structure of $G$. Note that no two induced $C_3$ graphs intersect, as each vertex is in at most one. As $G$ has no $C_3$ components then each of the induced $C_3$ must contain at least one degree three vertex. As $d_3=n(C_3)$, each induced $C_3$ contains exactly one degree three vertex and there are no other degree three vertices in the graph. Now all that remains are degree one and two vertices. Hence the other neighbour of each degree three vertex is either a leaf or a degree two vertex. It is easy to see that if it is a degree two vertex, this must be the beginning of a path of degree two vertices ending in a leaf, otherwise we would contradict either the component being unicyclic or the number of degree three or greater vertices. This shows that each component is either a cycle or a $D$-graph.
%{\bf ISN'T THIS ENOUGH TO SHOW THAT EACH COMPONENT IS EITHER A CYCLE OR A D GRAPH? IF SO, YOU CAN STATE THAT AND REMOVE REST OF PARAGRAPH} -------\textbf{YES. WASN'T SURE IF LEAVING OUT THE DETAILS MADE IT TOO SPARSE. WE WILL NEED TO KEEP AT LEAST THE LAST SENTENCE OF THIS PARAGRAPH TO TRANSITION TO NEXT PARAGRAPH.} Thus for every three vertex there is a degree one vertex, and as $d_1=d_3$ the only remaining vertices to add to $G$ are degree two. Note for each component $H_i$ currently in $G$, there is a $l\in \mathbb{N}$ such that $H_i \cong D_{l_i}$. Furthermore the only way to add vertices to $G$ is to create a 2-regular component, i.e. a cycle, or to lengthen the path of one or more of the $D_{l_i}$ components. 
As $D_{l_i} \sim C_{l_i}$, $G$ must be independently equivalent to a disjoint union of cycles.

Now let $G \sim C_{n_1} \cup C_{n_2} \cup \cdots \cup C_{n_r}$ for some $r \in \mathbb{N}$. Note each $n_j \geq 3$ as each component must have an equal number of vertices and edges. As the independence polynomial is multiplicative across components we have $i(G,x)= i(C_{n_1},x) \cdot i(C_{n_2},x) \cdots i(C_{n_r},x)$. It is easy to see from Theorem \ref{thm:PathPoly} and Proposition \ref{prop:deletion} that the leading coefficient and the coefficient of $x$ of $i(C_{n_j})$ are both $n_j$. Thus the leading coefficient of $i(G,x)$ is $n_1 \cdot n_2 \cdots n_r$ and the coefficient of $x$ is $n_1 + n_2 + \cdots +n_r$. However as $i(G,x) \sim i(C_n,x)$ then the leading coefficient and the coefficient of $x$ of $i(G,x)$ are both $n$. Thus $n_1 n_2 \cdots n_r=n_1 + n_2 + \cdots +n_r$. However a simple induction can show $n_1 \cdot n_2 \cdots n_r>n_1 + n_2 + \cdots +n_r$ for $r \geq 2$ and $n_j \geq 3$. As each $n_j\geq 3$ then $r=1$ and $G$ is connected. By Theorem \ref{thm:cycleconnectedequivclass}, we conclude that $[C_n] = \{C_n,D_n\}$.

\end{proof}

One notable exception to these results is $[C_{3^n}]$ when $n>1$. These cases are more difficult to deal with, as a graph in $[C_{3^n}]$ can have $C_3$ components which does not allow us the certainty of where the degree $3$ vertices are located among the components. We suspect that if $[C_n]$ grows large for certain $n$, then $n$ will be an odd multiple of $3$. For example, the only cycles that we know of with graphs other than $D_n$ and $C_n$ in their independence equivalence classes are $C_6$, $C_9$ and $C_{15}$. Oboudi showed in \cite{Oboudi2018} that $$[C_{9}]=\{C_{9},D_{9},G_1\cup C_3,G_2\cup C_3, G_3\cup C_3\}$$  where $G_1$, $G_2$, and $G_3$ are shown in Figure~\ref{fig:C9equivclass}.

\setcounter{subfigure}{0}
\begin{figure}[!h]
\def\c{0.6}
\def\r{1}
\centering
\subfigure[$G_1$]{
\scalebox{\c}{
\begin{tikzpicture}

\begin{scope}[every node/.style={circle,thick,fill,draw}]
    \node  (2) at (0.13397*\r,0.5*\r) {};
    \node (3) at (1*\r,0*\r) {};
    \node  (4) at (3*\r,0*\r) {};
    \node (7) at (0.13397*\r,-0.5*\r) {};
    \node (8) at (2*\r,0*\r) {};
    \node (9) at (-0.7*\r,-0.5*\r) {};   
\end{scope}

\begin{scope}

    \path [-] (2) edge node {} (3);
    \path [-] (3) edge node {} (7);
    \path [-] (4) edge node {} (8);
    \path [-] (2) edge node {} (7);
    \path [-] (3) edge node {} (8);
    \path [-] (9) edge node {} (7);
    
\end{scope}

\end{tikzpicture}}}
\qquad
\subfigure[$G_2$]{
\scalebox{\c}{
\begin{tikzpicture}

\begin{scope}[every node/.style={circle,thick,fill,draw}]
    \node (10) at (0.13397*\r,0.5*\r) {};
    \node (11) at (1*\r,0*\r) {};
    \node (12) at (3*\r,0*\r) {};
    \node  (15) at (0.13397*\r,-0.5*\r) {};
    \node (16) at (2*\r,0*\r) {};
    \node (17) at (-0.7*\r,0*\r) {};   
\end{scope}

\begin{scope}
    \path [-] (10) edge node {} (11);
    \path [-] (11) edge node {} (15);
    \path [-] (12) edge node {} (16);
    \path [-] (10) edge node {} (17);
    \path [-] (11) edge node {} (16);
    \path [-] (17) edge node {} (15);
    
\end{scope}
\end{tikzpicture}}}
\qquad
\subfigure[$G_3$]{
\scalebox{\c}{
\begin{tikzpicture}

\begin{scope}[every node/.style={circle,thick,fill,draw}]
    \node (18) at (9.13397*\r,0.5*\r) {};
    \node (19) at (10*\r,0.5*\r) {};
    \node (20) at (10.8*\r,0*\r) {};
    \node (21) at (11.8*\r,0*\r) {};
    \node (23) at (9.13397*\r,-0.5*\r) {};
    \node (24) at (10*\r,-0.5*\r) {};
\end{scope}

\begin{scope}
    \path [-] (18) edge node {} (19);
    \path [-] (20) edge node {} (21);
    \path [-] (19) edge node {} (20);
    \path [-] (20) edge node {} (24);
    \path [-] (23) edge node {} (24);
    \path [-] (18) edge node {} (23);
    
\end{scope}

\end{tikzpicture}}}
\caption{Components of the disconnected graphs in $[C_9]$}%
\label{fig:C9equivclass}%
\end{figure}
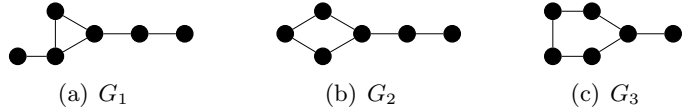

Computationally, we were able to show that  $$[C_{15}]=\{C_{15},D_{15},G_1\cup C_3\cup C_5,G_2\cup C_3\cup C_5,G_3\cup C_3\cup C_5\}$$ where $G_1'$, $G_2'$, and $G_3'$ are shown in Figure~\ref{fig:C15equivclass}.

\setcounter{subfigure}{0}
\begin{figure}[!h]
\def\c{0.6}
\def\r{1}
\centering
\subfigure[$G_1'$]{
\scalebox{\c}{
\begin{tikzpicture}

\begin{scope}[every node/.style={circle,thick,fill,draw}]
    \node  (2) at (0.13397*\r,0.5*\r) {};
    \node (3) at (1*\r,0*\r) {};
    \node  (4) at (3*\r,0*\r) {};
    \node (5) at (4*\r,0*\r) {};
    \node (7) at (0.13397*\r,-0.5*\r) {};
    \node (8) at (2*\r,0*\r) {};
    \node (9) at (-0.7*\r,-0.5*\r) {};   
\end{scope}

\begin{scope}

    \path [-] (2) edge node {} (3);
    \path [-] (4) edge node {} (5);
    \path [-] (3) edge node {} (7);
    \path [-] (4) edge node {} (8);
    \path [-] (2) edge node {} (7);
    \path [-] (3) edge node {} (8);
    \path [-] (9) edge node {} (7);
    
\end{scope}

\end{tikzpicture}}}
\qquad
\subfigure[$G_2'$]{
\scalebox{\c}{
\begin{tikzpicture}

\begin{scope}[every node/.style={circle,thick,fill,draw}]
    \node (10) at (0.13397*\r,0.5*\r) {};
    \node (11) at (1*\r,0*\r) {};
    \node (12) at (3*\r,0*\r) {};
    \node (13) at (4*\r,0*\r) {};
    \node  (15) at (0.13397*\r,-0.5*\r) {};
    \node (16) at (2*\r,0*\r) {};
    \node (17) at (-0.7*\r,0*\r) {};   
\end{scope}

\begin{scope}
    \path [-] (10) edge node {} (11);
    \path [-] (12) edge node {} (13);
    \path [-] (11) edge node {} (15);
    \path [-] (12) edge node {} (16);
    \path [-] (10) edge node {} (17);
    \path [-] (11) edge node {} (16);
    \path [-] (17) edge node {} (15);
    
\end{scope}
\end{tikzpicture}}}
\qquad
\subfigure[$G_3'$]{
\scalebox{\c}{
\begin{tikzpicture}

\begin{scope}[every node/.style={circle,thick,fill,draw}]
    \node (18) at (9.13397*\r,0.5*\r) {};
    \node (19) at (10*\r,0.5*\r) {};
    \node (20) at (10.8*\r,0*\r) {};
    \node (21) at (11.8*\r,0*\r) {};
    \node (23) at (9.13397*\r,-0.5*\r) {};
    \node (24) at (10*\r,-0.5*\r) {};
    \node (25) at (8.3*\r,0*\r) {};   
\end{scope}

\begin{scope}
    \path [-] (18) edge node {} (19);
    \path [-] (20) edge node {} (21);
    \path [-] (19) edge node {} (20);
    \path [-] (20) edge node {} (24);
    \path [-] (18) edge node {} (25);
    \path [-] (23) edge node {} (24);
    \path [-] (25) edge node {} (23);
    
\end{scope}

\end{tikzpicture}}}
\caption{Components of the disconnected graphs in $[C_9]$}%
\label{fig:C15equivclass}%
\end{figure}
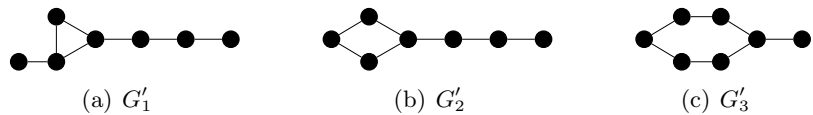

Despite the similarities between $[C_9]$ and $[C_{15}]$, we were able to computationally verify that $[C_{21}]=\{C_{21},D_{21}\}$ and $[C_{27}]=\{C_{27},D_{27}\}$.

\section{Concluding Remarks}\label{sec:conclusion}
Building on previous work in the literature, we explored the independence equivalence classes of paths and cycles. These were completely determined for connected graphs in \cite{Li2016} and \cite{Oboudi2018} respectively and our work extended this by considering disconnected graphs that belong to the independence equivalence classes. We showed that paths of odd length are independence unique, while paths of even length can have arbitrarily many graphs in their independence equivalence classes. For cycles, we showed that $[C_{n}]=\{C_{n},D_{n}\}$ when $n$ is even and not equal to $6$, or any power of a prime (the prime being at least $5$). We are left with some open problems.

\begin{problem}
What graphs can be in $[P_{2n}]$?
\end{problem}
We showed that $|[P_{2n}]|$ is unbounded, but this involved showing that $[P_{2n}]$ consisted of disjoint unions of cycles, graphs independence equivalent to cycles, and a path. However using the program Nauty \cite{McKay2014}, we were able to computationally determine that $|[P_{10}]|=10$. In addition to the $7$ graphs that we expect employing methods from the proof of Proposition~\ref{prop:largeevenpaths}, we found the $3$ surprising graphs in Figure~\ref{fig:P10equivclass}. What other graphs can belong to $[P_{2n}]$?

\setcounter{subfigure}{0}
\begin{figure}[!h]
\def\c{0.7}
\def\r{1}
\centering
\subfigure[$H_1$]{
\scalebox{\c}{
\begin{tikzpicture}

\begin{scope}[every node/.style={circle,thick,fill,draw}]
    \node (1) at (0*\r,0*\r) {};
    \node (2) at (1*\r,0*\r) {};
    \node (3) at (0.5*\r,1*\r) {};
   
    \node (4) at (2*\r,0*\r) {};
    \node (5) at (2*\r,1*\r) {};
    \node (6) at (3*\r,1*\r) {};
    \node (7) at (4*\r,1*\r) {};   
    \node (8) at (3*\r,0*\r) {}; 
    \node (9) at (4*\r,0*\r) {}; 
    \node (10) at (5*\r,0*\r) {};   
  
\end{scope}

\begin{scope}
\path [-] (2) edge node {} (3);
    \path [-] (1) edge node {} (2);
    \path [-] (1) edge node {} (3);
    
    \path [-] (4) edge node {} (5);
    \path [-] (5) edge node {} (6);
    \path [-] (6) edge node {} (7);
    \path [-] (6) edge node {} (8);
    \path [-] (8) edge node {} (9);
    \path [-] (9) edge node {} (10);
\end{scope}

\end{tikzpicture}}}
\qquad
\subfigure[$H_2$]{
\scalebox{\c}{
\begin{tikzpicture}
\begin{scope}[every node/.style={circle,thick,fill,draw}]
    \node (1) at (0*\r,0*\r) {};
    \node (2) at (1*\r,0*\r) {};
    \node (3) at (0.5*\r,1*\r) {};
   
    \node (4) at (2*\r,0*\r) {};
    \node (5) at (2*\r,1*\r) {};
    
    \node (6) at (3*\r,1*\r) {};
    \node (7) at (4*\r,1*\r) {};   
    \node (8) at (3*\r,0*\r) {}; 
    \node (9) at (4*\r,0*\r) {}; 
    \node (10) at (5*\r,0*\r) {};   
  
\end{scope}

\begin{scope}
    \path [-] (2) edge node {} (3);
    \path [-] (1) edge node {} (2);
    \path [-] (1) edge node {} (3);
    
    \path [-] (4) edge node {} (5);
    
    \path [-] (7) edge node {} (9);
    \path [-] (6) edge node {} (7);
    \path [-] (6) edge node {} (8);
    \path [-] (8) edge node {} (9);
    \path [-] (9) edge node {} (10);
\end{scope}
\end{tikzpicture}}}
\qquad
\subfigure[$H_3$]{
\scalebox{\c}{
\begin{tikzpicture}

\begin{scope}[every node/.style={circle,thick,fill,draw}]
    \node (1) at (0*\r,0*\r) {};
    \node (2) at (1*\r,0*\r) {};
    \node (3) at (0.5*\r,1*\r) {};
   
    \node (4) at (2*\r,0*\r) {};
    \node (5) at (2*\r,1*\r) {};
    
    \node (6) at (3*\r,1*\r) {};
    \node (7) at (4*\r,1*\r) {};   
    \node (8) at (3*\r,0*\r) {}; 
    \node (9) at (4*\r,0*\r) {}; 
    \node (10) at (5*\r,0.5*\r) {};   
  
\end{scope}

\begin{scope}
    \path [-] (2) edge node {} (3);
    \path [-] (1) edge node {} (2);
    \path [-] (1) edge node {} (3);
    
    \path [-] (4) edge node {} (5);
    
    \path [-] (7) edge node {} (9);
    \path [-] (6) edge node {} (7);
    \path [-] (7) edge node {} (10);
    \path [-] (8) edge node {} (9);
    \path [-] (9) edge node {} (10);
\end{scope}

\end{tikzpicture}}}
\caption{Surprising graphs in $[P_{10}]$.}
\label{fig:P10equivclass}%
\end{figure}
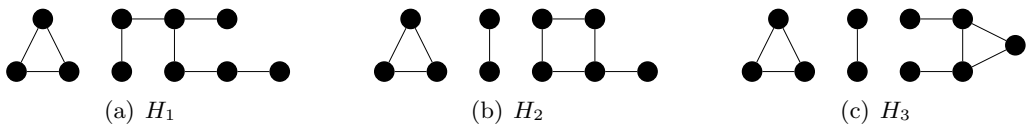

\begin{problem}
What graphs can be in $[C_{3n}]$?
\end{problem}
Multiples of $3$ make things more difficult when trying to characterize the equivalence classes of cycles as graphs in these classes can have triangle components. In fact, the only cycles we know of where $[C_n]\neq \{C_n,D_n\}$ are cycles with $n=3k$ for $k$ odd. Not every multiple of three has this property however, as $C_{21}$ is only equivalent to itself and $D_{21}$. Does $[C_{3n}]$ eventually stabilize to the two graphs we expect, or can it grow like the independence equivalence classes of even paths?

\begin{problem}
Are there families of graphs such that the independence equivalence class is unbounded \textbf{and} each independence polynomial is irreducible?
\end{problem}
We saw that $i(C_p,x)$ was irreducible and $|[C_{p}]|=2$ for all primes $p\ge 3$. An irreducible independence polynomial implies that all graphs in the independence equivalence class are connected. The restriction of connected graphs and irreducibility seems that it would make it less likely to have large independence equivalence classes, but the question remains open. We also think that studying the irreducibility of independence polynomials can be useful when studying independence equivalence classes of other graphs.

Finally, we leave the reader with a conjecture that all of our results  and computational work has lead us to believe is true.

\begin{conj}
If $3\!\!\!\not\vert n$ and $n\ge 4$, then $[C_{n}]=\{C_n,D_n\}$.
\end{conj}

%%%%%%%%%%%%%%%%%%%%%%%%%%%%%%%%%%%%%%%%%%%%%%%%%%%%%%%%%%%%%%%%%%%%%%%%%%%%%%%%%%%%%
\bibliographystyle{plain}
\bibliography{Trees}
\end{document}